\documentclass{amsart}

\usepackage{amssymb}
\usepackage{amscd}

\newtheorem{theorem}{Theorem}[section]
\newtheorem{proposition}[theorem]{Proposition}
\newtheorem{corollary}[theorem]{Corollary}

\newtheorem{fact}[theorem]{Fact}

\newtheorem{definition}[theorem]{Definition}
\newtheorem{remark}[theorem]{Remark}
\newtheorem{example}[theorem]{Example}

\providecommand{\rk}{\mathop{rank}}
\providecommand{\pw}{\mathop{pw}}
\providecommand{\pr}{\mathop{PF}}

\providecommand{\Th}{\mathop{Th}}
\providecommand{\sgn}{\mathop{sign}}

\newcommand\CC{\mathcal C}
\newcommand\CF{{\mathcal F}}

\newcommand\CL{{\mathcal L}}
\newcommand\CM{{\mathcal M}}

\newcommand\CP{{\mathcal P}}

\newcommand\CU{{\mathcal U}}
\newcommand\BN{{\mathbb N}}
\newcommand\BZ{{\mathbb Z}}
\newcommand\BQ{{\mathbb Q}}

\newcommand\BF{{\mathbb F}}
\newcommand\fm{\mathfrak{m}}
\newcommand\fp{\mathfrak{p}}

\newcommand\x{\bar{x}}
\newcommand\y{\bar{y}}
\newcommand\z{\bar{z}}
\newcommand\g{\bar g}
\newcommand\ba{\bar{a}}
\newcommand\bb{\bar{b}}

\newcommand{\twopartdef}[4]
{
	\left\{
		\begin{array}{ll}
			#1 & \mbox{if } #2 \\
			#3 & \mbox{if } #4
		\end{array}
	\right.
}

\title{Some arithmetic properties on nonstandard rationals}

\author{Junguk \textsc{Lee}}
\address{Department of Mathematics\\ Yonsei University\\
50 Yonsei-Ro, Seodaemun-Gu\\
Seoul 03722, Korea}
\thanks{The author thanks Anand Pillay for his suggestions and comments in Appendix.}

\begin{document}
\maketitle

\begin{abstract}
For a given number field $K$, we show that the ranks of nonsingular elliptic curves over $K$ are uniformly finitely bounded if and only if weak Mordell-Weil property holds in all(some) ultrapowers $^*K$ of $K$. Also we introduce nonstandard Mordell-Weil property for $^*K$ considering each Mordell-Weil group as $^*\BZ$-module, where $^*\BZ$ is an ultrapower of $\BZ$, and we show that nonstandard Mordell-Weil property is equivalent to weak Mordell-Weil property in $^*K$. In Appendix, we showed that it is possible to consider definable abelian groups as $^*\BZ$-modules in a saturated nonstandard rational number field $^*\BQ$ so that nonstandard Mordell-Weil property is well-defined, and thus we showed that nonstandard Mordell-Weil property and weak Mordell-Weil property are equivalent. 

Next we focus on priems and prime ideals of nonstandard raional number fields. We give an infinite factorization theorem on $^*\BQ$ using valuations induced from primes of $^*\BZ$, and we classify maximal and prime ideal of $^*\BZ$ in terms of maximal filter on the set of primes of $^*\BZ$ and ordered semigroups of the valuation semigroup induced from maximal ideals of $^*\BZ$.
\end{abstract}

\section{Introduction}
In this note, we see some arithmetic properties of nonstandard number fields, which are ultrapowers of a given number field. At first, we look at elliptic curves over nonstandard number fields and most of all we are interested in the ranks of elliptic curves. The rank of a elliptic curve $E$ on a given field $K$ is an important invariant to measure the size of $K$-rational points $E(K)$. The $K$-rational points $E(K)$ forms an abelian group, called Mordell-Weil group and so $E(K)\otimes_{\BZ} \BQ$ forms a $\BQ$-vector space. The dimension of this vector space is called the rank of $E(K)$, denoted by $\rk E(K)$. The ranks of elliptic curves over global fields like number fields or finite extensions of function fields over finite fields are finite by Mordell-Weil Theorem. One can ask how the ranks of elliptic curves over a global field can be large. It has a negative answer for the case of function fields. In \cite{ST}\cite{U} it's shown that the rank of elliptic curve can be arbitrary large in $\BF_p(t)$. But it is not known much about the boundedness of ranks of elliptic curves over number field. In \cite{RS1} it's shown that the quadratic twists $E_d$ of a elliptic curve $E$ over $\BQ$ have a bounded rank if and only if a series associated with $E$ is convergent and in \cite{BS}, the average rank of elliptic curves over $\BQ$ has a finite value. Here, we show that weak Mordell-Weil properties of $\aleph_1$-saturated nonstandard number fields imply that the ranks of all elliptic curves over a given number field are uniformly finitely bounded.

Next we focus on primes and prime ideals in the nonstandard rational number fields. A nonstandard rational number field $^*\BQ$ has a nonstandard integer ring $^*\BZ$ corresponding to $\BZ$ in $\BQ$. It satisfies some basic arithmetic properties of $\BZ$ : Its field of fractions is $^*\BQ$, it is integrally closed in $^*\BQ$, and the units in $^*\BZ$ are only $\pm 1$. Unfortunately $^*\BZ$ need not be a Dedekind domain and not Noetherian any more. But $^*\BZ$ has the set of primes and each prime gives a valuation on $^*\BQ$. Using these valuations we can identify elements in $^*\BQ$, and thus we get an infinite version of factorization. As a consequence, $^*\BZ$ is the intersection of valuation rings of each valuation induced from primes. Even though $^*\BZ$ is not a principal ideal domain, any finitely generated ideal is a principal ideal. Combining this property and infinite factorization, we classify maximal ideals in terms of filters on the set of primes in $\BZ$. From the classification of prime ideals in ultraproducts of Dedekind domains in \cite{OS}, we classify prime ideals of $^*\BZ$ in terms of valuation semigroups induced from maximal ideals.
\section{Uniformly finite boundedness of the ranks of elliptic curves}
We start with a review of elliptic curves over a field $K$ of zero characteristic. A nonsingular elliptic curve $E$ over $K$ is given by the following equation \begin{equation*}
y^2=x^3+Ax+B, 4A^3+27B^2\neq 0
\end{equation*}
for $A,B\in K$. It is well-known that the set $E(K)$ of $K$-rational points of $E$ 

\begin{equation*}
\{(x,y)\in K^2|\ y^2=x^3+Ax+B\}\cup\{P_{\infty}\}
\end{equation*}
forms an abelian group called Mordell-Weil group. If $K$ is a global field for example a number field or a finite extension of function field over $\BF_q$, then Mordell-Weil Theorem says that $E(K)$ is finitely generated and its rank is finite. Now let $K$ be a global field. Then by Mordell-Weil Theorem, the information of rank of $E(K)$ is contained in the cardinality of {\em weak $n$th Mordell-Weil group}. Let $n\ge 2$ and let 
\begin{equation*}
nE(K):=\{P\in E(K):\ \exists Q\in E(K),\ P=\overbrace{Q+_E\cdots+_E Q}^{n}\}.
\end{equation*}
The {\em weak $n$th Mordell-Weil group} of $E$ over $K$ is the quotient of $E(K)$ by $nE(K)$, denoted by $E/nE(K)$. Then $n^r\le |E/nE(K)|\le n^r+n^2$ where $r$ is the rank of $E$ over $K$ for all $n\ge 2$. From now on, if there is no confusion on $K$, we omit 'over $K$'. We consider only nonsingular elliptic curves. We say a field $L$ has {\em Weak Mordell-Weil property} if for any elliptic curve $E$ over $L$, each weak $n$th Mordell-Weil group of $E$ over $L$ is finite.\\

\medskip

We may consider elliptic curves as definable objects in $K$. Fix $A,B\in K$ such that $4A^3+27B^2\neq 0$. Then $E(K)$ can be seen as a definable subset of $K^3$ by the following formula
\begin{equation*}
E(A,B;x,y,z)\equiv [(y^2=x^3+Ax+B)\wedge z=1]\vee [x=0\wedge y=1\wedge z=0]\wedge (4A^3+27B^2\neq0).
\end{equation*}
Moreover the group operation $+_E$ of the Mordell-Weil group $(E(K),+_E)$ is also definable, that is, the graph of $+_E:E(K)\times E(K)\rightarrow E(K)$ is a definable subset of $K^3\times K^3\times K^3$ given by the following formula : For $\x=(x_0,x_1,x_2)$, $\y=(y_0,y_1,y_2)$, and $\z=(z_0,z_1,z_2)$,
\begin{equation*}
\begin{array}{c c l}
+_E(\x,\y,\z)&\equiv&E(\x)\wedge E(\y)\wedge E(\z)
\wedge\\
&&[(x_2=0\rightarrow \bigwedge_{0\le i\le 2}\limits y_i=z_i)\\
&\vee&(y_2=0\rightarrow \bigwedge_{0\le j\le 2}\limits x_j=z_j)\\
&\vee&(x_0=y_0\wedge x_1\neq y_1\rightarrow (z_0=0\wedge z_1=1\wedge z_2=0))\\
&\vee&(x_0=y_0\wedge x_1= y_1 \rightarrow (z_0=(\frac{3x_0^2+A}{2y_0})^2-x_0-x_1)\wedge\\
&&(z_1=-\frac{3x_0^2+A}{2y_0}z_0-(x_1-\frac{3x_0^2+A}{2y_0}x_0))\wedge\\
&&(z_2=1))\\
&\vee&(x_0\neq y_0\rightarrow(z_0=(\frac{y_1-x_1}{y_0-x_0})^2-x_0-x_1)\wedge\\
&&(z_1=\frac{y_1-x_1}{y_0-x_0}z_0-\frac{y_0x_1-y_1x_0}{y_0-x_0})\wedge\\
&&(z_2=1))].
\end{array}
\end{equation*}
For $n\ge 2$, each $nE(K)$ is definable. Consider a formula 
\begin{equation*}nE(\x)\equiv E(\x)\wedge \exists \x_1,\x_2,\ldots,\x_{n}\ (\bigwedge_{1\le i< n}\limits (E(\x_i)\wedge +_E(\x_i,\x_1,\x_{i+1}))\wedge \x_n=\x)
\end{equation*}
and this formula defines $nE(K)$.\\

\medskip

We recall basic properties of the notion of ultraproduct. Fix a countably infinite index set $I$ and a nonprincipal ultrafilter $\CU$ on $I$. Let $\{\CM_i\}_{i\in I}$ be a set of $\CL$-structures by the set $I$. Take the ultraproduct $\Pi_{\CU} \CM_i:=\Pi_{i\in I} \CM_i/\sim_{\CU}$ of $\CM_i$'s with respect to the ultrafilter $\CU$, where $(a_i)\sim_{\CU}(b_i)$ if and only if $\{i\in I|a_i=b_i\}\in \CU$. We denote by $(a_i)$ the element in $\Pi_{\CU} \CM_i$ given by the equivalence class of $(a_i)$. A subset $S$ of $\Pi_{\CU}\CM_i$ is called {\em induced} if it is of the form of $(S_i)/\sim_{\CU}$ for some $S_i\subset \CM_i$.
\begin{remark}
Let $\{\CM_i\}_{i\in I}$ be a collection of infinite structures indexed by $I$.
\begin{enumerate}
	\item For a formula $\phi(x_1,\ldots,x_n)$ and $a^1=(a_i^1),\ldots, a^n=(a_i^n)\in \Pi_{\CU} \CM_i$,
\begin{equation*}
\Pi_{\CU}\CM_i\models \phi(a^1,\ldots,a^n)\ \Leftrightarrow\ \{i\in I|\ \CM_i\models \phi(a_1^i,\ldots,a_n^i)\}\in \CU.
\end{equation*}

	\item The ultraproduct $\Pi_{\CU}\CM_i$ is $\aleph_1$-saturated.
	
	\item An induced set $\Pi_{\CU}\CM_i$ is finite or at least $\ge 2^{\aleph_0}$, and any definable set is induced.
\end{enumerate}
\end{remark}
\noindent For a fixed infinite structure $\CM$, if $\CM_i=\CM$ for $i\in I$, we write $^*\CM_{\CU}$ for the ultrapower of $\CM$ with respect to the ultrafilter $\CU$. We write $^*\CM$ if $\CU$ is obvious. In this case, there is a canonical embedding $\iota$ from $\CM$ to $^*\CM$ and this embedding is an elementary embedding, that is, for a formula $\phi(\x)$ with $|\x|=n$ and $\ba\in \CM^n$, $\CM\models \phi(\ba)$ if and only if $^*\CM\models \phi(\iota(\ba))$. So $^*\CM$ is an $\aleph_1$-saturated elementary extension of $\CM$.

At first, we see some properties of ultraproducts of abelian groups. Let $\{A_i\}_{i\in I}$ be a set of abelian groups and consider the ultraproduct $\Pi_{\CU}A_i$. Then we can consider the ultraproduct $\Pi_{\CU}A_i$ as a $^*\BZ$-module, where $^*\BZ$ is the ultrapower $\Pi_{\CU}\BZ$ of $\BZ$ as follows : For $a=(a_i)\in \Pi_{\CU}A_i$ and $n=(n_i)\in \Pi_{\CU}\BZ$, define $na:=(n_ia_i)$.
\begin{remark}\label{ultraproduct_abeliangroup}
Let $\{A_i\}_{i\in I}$ be a set of abelian groups indexed by $I$.
\begin{enumerate}
	\item If each $A_i$ is generated by $n$-elements, then the ultraproduct $\Pi_{\CU}A_i$ is generated by $n$-many elements as $^{*}\BZ$-module.
	\item If the ultraproduct $\Pi_{\CU}A_i$ is finteley generated as $^*\BZ$-module, then the cardinality of the quotient of $\Pi_{\CU}A_i$ by $k\Pi_{\CU}A_i$ is finite for all $k\ge 1$. More precisely, if $\Pi_{\CU}A_i$ is generated by $n$-elements, then $|\Pi_{\CU}A_i/k\Pi_{\CU}A_i|\le k^n$ for each $k\ge 1$.
\end{enumerate}
\end{remark}
\begin{proof}
(1) Suppose $A_i$ is generated by $a_{i1},\ldots,a_{in}$ for each $i\in I$. Then $\Pi_{\CU}A_i$ is generated by $a_1=(a_{i1}),\ldots,a_n(a_{in})$ as $^*\BZ$-module. Take $x=(x_i)\in \Pi_{\CU}A_i$ arbitrary. For each $i\in I$, there are $c_{1i},\ldots,c_{ni}$ in $\BZ$ such that $x_i=c_{1i}a_{i1}+\cdots+c_{ni}a_{in}$ so that $x=c_1a_1+\cdots+c_na_n$ for $c_1=(c_{1i}),\ldots,c_n=(c_{ni})\in {^*\BZ}$. Thus $\Pi_{\CU}A_i$ is generated by $n$-elements as $^*\BZ$-module.\\

(2) Suppose $\Pi_{\CU}A_i$ is generated by $a_1,\ldots,a_n$ as $^*\BZ$-module for some $n\ge 1$. Let $k\ge 1$. Define a map $f_k:{^*\BZ/k{^*\BZ}}\times\cdots\times {^*\BZ/k{^*\BZ}}\rightarrow \Pi_{\CU}A_i/k\Pi_{\CU}A_i$ by sending $(c_1,\ldots,c_n)$ to $c_1a_1+\cdots+c_na_n$. Then this map is well-defined and it is onto. Since $^*\BZ/k{^*\BZ}\cong \BZ/k\BZ$, the domain of $f_k$ is finite and its cardinality is $k^n$. Thus, the cardinality of $\Pi_{\CU}A_i/k\Pi_{\CU}A_i$ is less than or equal to $k^n$.
\end{proof}
\noindent Note that $E(K)$ is an ultraproduct of elliptic curves over $K$ for a ultrapower $^*K$ of a field $K$ and an elliptic curve $E$ over $^*K$. Thus $E({^*K})$ is an ultraproduct of abelian groups.\\  

\medskip

From now on, we fix a number field $K$. We see some relations between the ranks of elliptic curves over $K$ and over $^*K$. Let $E$ be an elliptic curve over $K$, then $E$ is also over $^*K$ and $E(K)\subset E(^*K)$. So, it can be directly shown that the rank of $E({^*K})$ is equal to or larger than the rank of $E(K)$. But unfortunately the rank of elliptic curve is not an elementary invariant, that is, for an elliptic curve $E$ over $K$, the rank of $E(K)$ need not be equal to the rank of $E({^*K})$ unless $\rk E(K)=0$.
\begin{theorem}\label{rank_not_elementary}
If $\rk E(K)$ is not equal to $0$, then $\rk E({^*K})>\rk E(K)$ and $\rk E({^*K})$ is always infinite. If $\rk E(K)=0$, then $E(K)=E({^*K})$.
\end{theorem}
\begin{proof}
Let $E$ be an elliptic curve over $K$. Suppose $\rk E(K)>0$ so that $E(K)$ is infinite. Thus $E({^*K})$ is infinite. Since $E({^*K})$ is induced and infinite, the cardinality of $E({^*K})$ is $2^{\aleph_0}$. Also there are countably many torsion points in $E({^*K})$. Thus the vector space $E({^*K})\otimes \BQ$ is of cardinality of $2^{\aleph_0}$. So its dimension as $\BQ$-vector space is $2^{\aleph_0}$ and $\rk E({^*\BZ})=2^{\aleph_0}$.

%
%

Suppose $\rk E(K)=0$ so that $E(K)$ is finite. Let $k=|E(K)|$. We can write down in the sentence $\phi$ saying there are only $k$-many points in $E$. Since $K$ and $^*K$ are elementary equivalent, $^*K\models \phi$, that is, $k=|E({^*K})|$. Always $E(K)$ is a subset of $E({^*K})$ and therefore $E(K)$ and $E({^*K})$ are same.
\end{proof} 
\noindent From Proposition \ref{rank_not_elementary}, the rank itself may not be an elementary invariant. By the way each weak $n$th group may be a good elementary invariant. Consider an equivalence relation $\sim_{E,n}$ on $E$ defined by the following formula \[\exists \z\ (nE(\z)\wedge +_E(\x,\z,\y))\] for each elliptic curve $E$ and $n\ge 2$.
\begin{proposition}\label{weak-nth_elementary}
Let $E$ be an elliptic curve over $K$. For any $n\ge 2$, $E(K)/nE(K)=E({^*K}/nE({^*K})$.
\end{proposition}
\begin{proof}
Let $E$ be an elliptic curve over $K$ with $|E(K)/nE(K)|=k_n<\infty$ for $n\ge 2$. Fix $n\ge 2$. There is a natural embedding $\iota_E$ from $E(K)$ to $E({^*K})$ and it induces a map $\iota_{E,n}$ from $E(K)/nE(K)$ to $E({^*K})/nE({^*K})$. Since $E(K)\cap nE({^*K})=nE(K)$, this map is injective. It remains to show surjectivity. For $n\ge 2$ and $k\ge 1$, define a formula 
\begin{equation*}
\begin{array}{c c l}
\phi_{E,n,k}(\x_1,\ldots,\x_{k})&\equiv&\bigwedge_{1\le i\le k}\limits E(\x_i)\\
&\wedge&\bigwedge_{1\le i<j\le k}\limits \neg(\x_i\sim_{E,n}\x_j).
\end{array}
\end{equation*}
Next consider the following sentence
\begin{equation*}
\begin{array}{c c l}
\phi_{E,n}' &\equiv& \exists \x_1\ldots\x_{k_n}[\phi_{E,n,k_n}(\x_1,\ldots,\x_{k_n})\\
&\wedge& \forall \x\ (E(\x)\rightarrow \bigvee_{1\le i\le k_n}\limits (\x\sim_{E,n}\x_i))].
\end{array}
\end{equation*}
Since $|E(K)/nE(K)|=k_n<\infty$, $K\models \phi_{E,n}'$, and so $^*K \models \phi_{E,n}'$. Thus $|E/nE({^*K})|=k_n$ also. Thus $\iota_{E,n}$ is surjective and it is bijective. Therefore $\iota_{E,n}:E/nE(K)\cong E/nE({^*K})$.
\end{proof}
\noindent It looks good to use a notion of weak $n$th Mordell-Weil groups rather than a notion of rank of elliptic curve to see the relation between ranks of elliptic curve over $K$ and over $^*K$. We get the following equivalent conditions for the boundedness of ranks of elliptic curves over $K$.
\begin{definition}
We say $^*K$ has nonstandard Mordell-Weil property if each Mordell-Weil group of elliptic curve over $^*K$ is finitely generated as $^*\BZ$-module. 
\end{definition}
\begin{theorem}\label{rkbound_wMW}
The followings are equivalent :
\begin{enumerate}
	\item The ranks of elliptic curves over $K$ are uniformly finitely bounded.
	\item For each $n\ge 2$, the cardinalities of weak $n$th Mordell-Weil groups over $K$ are uniformly finitely bounded.
	\item The cardinalities of weak $2$nd Mordell-Weil groups over $K$ are uniformly finitely bounded.
	\item For any nonprincipal ultrafilter $\CU$ on $I$, $^*K_{\CU}$ has Weak Mordell-Weil property.
	\item For some nonprincipal ultrafilter $\CU$ on $I$, $^*K_{\CU}$ has Weak Mordell-Weil property.
	\item Weak $2$nd Mordell-weil groups of elliptic curves over some $^*K$ are finite.
	\item Nonstandard Moredell-Weil property hold for all $^*K$.
	\item Nonstandard Moredell-Weil property hold for some $^*K$.
\end{enumerate}
\end{theorem}
\begin{proof}
It is easy to check $(1)\Leftrightarrow(2)\Leftrightarrow (3)$, $(4)\Rightarrow(5)\Rightarrow(6)$, and $(7)\Rightarrow (8)$. It it enough to show $(3)\Rightarrow (4)$, $(6)\Rightarrow (3)$, $(1)\Rightarrow (7)$, and $(8)\Rightarrow (5)$. Let $E(A,B;x,y,z)$ be a formula 
\begin{equation*}
(z=1\wedge y^2=x^3+Ax+B)\vee(z=0\wedge x=0\wedge y=1)\wedge (-16(4A^3+27B^2)\neq 0),
\end{equation*}
which parametrizes all pairs of nonsingular elliptic curves $E(A,B):y^2=x^3+Ax+B$ and points in $E(A,B)$. Consider a two variable formula $\Phi_{n,m}(A,B)\equiv \exists \x_1,\ldots,\x_m\phi_{E(A,B),n,m}$ for each $m\ge 1$ which parametrizes all nonsingular elliptic curves whose the weak $n$th Mordell-Weil group has the cardinality at least $m$.\\

\smallskip

$(3)\Rightarrow (4)$. Suppose $(3)$ holds. Choose $n\ge 2$.Then there is $M_n >0$ such that $\BQ\models \forall A,B\ \neg \Phi_{n,m}(A,B)$ for all $m\ge M_n$. So for all $^*K\equiv K$, $^*K\models \forall A,B\ \neg \Phi_{n,m}(A,B)$ for all $m\ge M_n$ and any weak $n$th Mordel-Weil group over $^*K$ has cardinality less than $M_n$. So $(4)$ holds.\\

\smallskip

$(6)\Rightarrow (3)$. Suppose $(3)$ does not hold. Consider the countable set of formulas $\Phi_{2,m}(A,B)$. For each $m>0$, $\Phi_{2,m}(K)\neq\emptyset$. So any nonstandard rational number field $^*K$ contains an elements $({^*A},{^*B})$ such that for $^*K\models \Phi_{2,m}({^*}A,{^*}B)$ for any $m>0$ and $(6)$ is not true.\\

\smallskip

$(1)\Rightarrow (7)$. Suppose there is $C>0$ such that $\rk E(K)<C$ for each elliptic curve $E$ over $K$. There is $r>0$ such that $E(K)$ is generated by at most $r$-many elements for all elliptic curves $E$ over $K$ because there are only finitely many possibilities for torsion points of elliptic curves over $K$. Then by Remark \ref{ultraproduct_abeliangroup} (1), each Mordell-Weil group of elliptic curves over $^*K$ is finitely generated as $^*\BZ$-module.\\

\smallskip

$(8)\Rightarrow (5)$. Suppose nonstandard Mordell-Weil property hold for some $^*K$. By Remark \ref{ultraproduct_abeliangroup} (2), each $n$th weak Mordell-Weil group of elliptic curves over $^*K$ is finite and so weak Mordell-Weil property holds for $^*K$.
\end{proof}
%

\section{Factorization in $^*\BZ$}
In \cite{R}, it was first noted that the integer ring $\BZ$ is definable in $\BQ$. Let $N(x)$ be a formula defining $\BZ$ in $\BQ$. Then $^*\BQ$ also has a nonstandard integer ring $^*\BZ:=Z({^*}\BQ)$ corresponding to $\BZ$ in $\BQ$ as a definable subset in $^*\BQ$. By Lagrange theorem, the set of natural numbers $\BN$ is definable in $\BQ$ by the formula 
\begin{equation*}
N'(x)\equiv \exists y_1,y_2,y_3,y_4(\bigwedge_{1\le i\le 4}\limits Z(y_i)\wedge x=y_1^2+y_2^2+y_3^2+y_4^2),
\end{equation*}
and $^*\BQ$ also has the nonstandard natural numbers $^*\BN:=N({^*}\BQ)$. Then $^*\BZ$ inherits some basic arithmetic properties of $\BZ$ : The quotient field of $^*\BZ$ is $^*\BQ$, it is integrally closed in $^*\BQ$, its units are only $\pm 1$, and any finitely generated ideal is a principal ideal. Unlike $\BZ$, the nonstandard integer ring $^*\BZ$ need not be a Dedekind domain and not Noetherian so that $^*\BZ$ need not be a PID. 
\begin{example}
Let $^*\BQ$ be $\aleph_1$-saturated. Let $I=\bigcap_{i\in\omega}2^i{^*}\BZ$. Then $I$ is not a finite product of prime ideals and it is not finitely generated.
\end{example}
\noindent So (finite) factorization theorem does not hold for $^*\BZ$ any more. But each primes in $^*\BZ$ gives a valuation, and using these valuations we'll get an infinite version of factorization.

We start with defining a binary relation by a formula $x|y\equiv Z(x)\wedge Z(y)\wedge \exists z(Z(z)\wedge z\neq 0\wedge y=zx)$ saying that $x$ and $y$ are integers and $x$ divides $y$. Next we consider a formula defining primes as follow :
\begin{equation*}
P(x)\equiv N(x)\wedge x\neq 1 \wedge (\forall y\ y|x\rightarrow (y=\pm 1 
\vee y=\pm x)).
\end{equation*}
So, $P(\BQ)$ gives the set $\CP(\BN)$ of primes in $\BN$ and $P({^*}\BQ)$ is the set $\CP({^*}\BN)$ of primes in $^*\BZ$ corresponding $\CP(\BN)$. For each prime, we define the set of powers of a given prime. Consider a formula 
\begin{equation*}
\pw(x;y)\equiv N(x)\wedge N(y)\wedge(y=1 \vee \forall z(N(z)\wedge z\neq \pm1\wedge  z|y\rightarrow x|z)).
\end{equation*}
For each $p\in\CP({^*}\BN)$ let $p^{^*\BN}:=\pw(p;{^*}\BN)$ be the set of $1$ and elements in $^*\BN$ divisible by only $p$ and consider the function $d_p: {^*}\BZ\setminus\{0\}\rightarrow p^{^*\BN}$ sending each $x$ to $y$ such that $y|x$ but $\neg(py|x)$, which is definable by a formula
\begin{equation*}
d_p(x,y)\equiv Z(x)\wedge \pw(p;y)\wedge y|x\wedge \neg(py|x).
\end{equation*}
For a valuation induced from $p\in\CP({^*}\BN)$, we have to extend $d_p$ to $^*\BQ\setminus\{0\}$. Let $\pw'(x;y)\equiv \pw(x;y)\vee \exists z (\pw(x;z)\wedge zy=1)$. Let $p^{^*\BZ}:=\pw'(p;{^*}\BQ)=p^{^*\BN}\cup\{1/x|x\in p^{^*\BN}\}$. Note that after ordering $<_p$ on $p^{^*\BZ}$ as $a<_p b$ if $b/a\in p^{^*\BN}$, $(p^{^*\BZ},\times,<_p)$ forms an ordered group. Extend $d_p$ to the map from $^*\BQ\setminus\{0\}$ into $p^{^*\BZ}$ by sending $a/b$ with $a,b\in {^*}\BZ$ into $d_p(a)/d_p(b)$, which is a surjective group homomorphism from $(^*\BQ\setminus\{0\},\times)$ to $(p^{^*\BZ},\times)$. Then for $a,b\in {^*\BQ}\setminus\{0\}$ with $a+b\neq 0$, $d_p(a+b)\ge \min(d_p(a),d_p(b))$. Thus $d_p:\ {^*}\BQ\setminus \{0\}\rightarrow p^{^*\BZ}$ is a valuation, which is definable by a formula
\begin{equation*}
\begin{array}{c c l}
d_p'(x,y)&\equiv& x\neq 0\wedge \exists x_1x_2y_1y_2[x_2x=x_1\\
&\wedge& d_p(x_1,y_1)\wedge d_p(x_2,y_2)\\
&\wedge& y_2y=y_1].
\end{array}
\end{equation*}
We also have a Euclidean absolute value in $^*\BQ$. Consider an order $<$ on $^*\BZ$ as $a<b$ if $b-a\in {^*}\BN$ and extend $<$ on $^*\BQ$ by defining $a_1/b_1<a_2/b_2$ with $a_1,a_2\in {^*}\BZ$ and $b_1,b_2\in {^*}\BN$ if $a_2b_1-a_1b_2\in {^*}\BN$. This ordering is definable by a formula
\begin{equation*}
\begin{array}{c c l}
x<y&\equiv& \exists x_1x_2y_1y_2\ [Z(x_1)\wedge Z(y_1)\wedge N(x_2)\wedge N(y_2)\\
&\wedge& xx_2=x_1\wedge yy_2=y_1\\
&\wedge& N(y_1x_2-x_1y_2) ].
\end{array}
\end{equation*}
Using these valuations induced from primes and the Euclidean absolute value, we get the following factorization.
\begin{theorem}\label{factorization}
For $a_0,a_1\in {^*}\BQ\setminus\{0\}$,\begin{center}
 $a_0=a_1$ iff $d_p(a_0)=d_p(a_1)$ for all $p\in \CP({^* \BN})$ and $a_0a_1>0$.
\end{center}
In other words, there is a injective group homomorphism from ${^*}\BQ\setminus\{0\}$ to $\Pi_{p\in \CP({^*}\BN)} p^{^*\BZ}\times \{\pm 1\}$ sending $a\mapsto (d_p(a),\sgn(a))$, where $\sgn$ is the projection map from ${^*}\BQ\setminus\{0\}$ to ${^*}\BQ\setminus\{0\}/{^*}\BQ^{>0}\cong \{\pm 1\}$ and ${^*}\BQ^{>0}:=\{r\in {^*}\BQ|r>0\}$. 
\end{theorem}
\begin{proof}
Theorem \ref{factorization} holds for $\BQ$. So, the following sentence
\begin{equation*}
\begin{array}{c c l}
\phi&\equiv&\forall xy[(xy\neq 0\\
&\wedge& \forall z( P(z)\rightarrow \forall w(d_z(x,w)\leftrightarrow d_z(y,w)) )\\
&\wedge& (x>0\leftrightarrow y>0) )\\
&\rightarrow& x=y]
\end{array}
\end{equation*}
is true in $\BQ$, and so is in $^*\BQ$ because $\BQ\equiv {^*}\BQ$.
\end{proof}
\noindent As corollary, $^*\BZ$ is the intersection of all $d_p$-valuation rings for $p\in\CP({^*}\BN)$.
\begin{corollary}\label{element_primelocal}
Let $^*\BZ_p$ be the $d_p$-valuation ring for $p\in\CP({^*}\BN)$. Then,
\begin{equation*}
{^*\BZ}=\bigcap_{p\in\CP({^*}\BN)}\limits {^*}\BZ_p.
\end{equation*}
\end{corollary}
\begin{remark}
From Corollary \ref{element_primelocal}, it can be expected that for two ideals $I_1$ and $I_2$ in $^*\BZ$, if $I_1{^*\BZ_p}=I_2{^*\BZ_p}$ for all $p\in\CP({^*\BN})$, then $I_1=I_2$ but unfortunately this fails. We will see later that there is a maximal ideal $\fm$ in $^*\BZ$ which is not contained in $p{^*\BZ}$ for any $p\in \CP({^*\BN})$. This maximal ideal satisfies $\fm{^*\BZ_p}={^*\BZ_p}$ for all $p\in \CP({^*\BN})$ but $\fm\neq {^*\BZ}$.
\end{remark}

\section{Prime ideals in $^*\BZ$}
In this section, we study the prime ideals in $^*\BZ$.
Let $I$ be a countable index set and let $\CU$ be a nonprincipal ultrafilter. From now on, we fix an ultrapower $^*\BQ:=\Pi_{\CU}\BQ$ of $\BQ$ so that $^*\BQ$ is a nonstandard rational number field which is $\aleph_1$-saturated, and the nonstandard integer ring $^*\BZ$ of $^*\BQ$ is also the ultrapower of $\BZ$. At first, we describe maximal ideals in $^*\BZ$. We define {\em a prime factor} for each elements in $^*\BZ$.
\begin{definition}\label{primefactor}
Let $n\in {^*\BZ}$. {\em The prime factor $\pr(n)$ of $n$} is the set of primes dividing $n$, that is, $\pr(n):=\{p\in\CP({^*}\BN)|\ p|n\}$. Denote $\pr({^*\BZ})$ for the set of all prime factors of elements in $^*\BZ$. Note that $\pr(\pm 1)=\emptyset$, $\pr(0)=\CP({^*\BZ})$, and $\emptyset\subsetneq \pr(n)\subsetneq \CP({^*\BZ})$ for $n\neq \pm 1, 0$.
\end{definition}
\begin{definition}\label{filter_on_primefactor}
\begin{enumerate}
	\item We say a subset $\mathcal{F}$ of $\pr({^*\BZ})$ is {\em a filter on $\pr({^*\BZ})$} if
	\begin{enumerate}
		\item $\emptyset \notin \mathcal{F}$;
		\item for $\pr(n),\pr(m)\in \mathcal{F}$, $\pr(n)\cap\pr(m)$ is in $\mathcal{F}$; and
		\item for $\pr(n)\in F$ and $S\subset\CP({^*\BN})$, if $\pr(n)\subset S$ and $S=\pr(m)$ for some $m$, then $S\in \mathcal{F}$.
	\end{enumerate}
	\item A filter $\mathcal{F}$ is {\em maximal} if there is no filter properly containing $F$.
	\item A filter $\mathcal{F}$ is {\em principal} if there is $n\in{^*\BZ}$ such that for $S\in\pr({^*}\BZ)$
	\begin{center}
	$S\in F$ if and only if $\pr(n)\subset S$.
	\end{center}
In this case, we write $\mathcal{F}=\mathcal{F}(n)$.
\end{enumerate}
\end{definition}
\begin{remark}\label{gcd}
For $n,m\in{^*}\BZ$, there is $l\in{^*}\BN$ such that $n{^*}\BZ+m{^*}\BZ=l{^*}\BZ$ and such an element is uniquely determined. We call such $l$ a {\em great common divisor of $n$ and $m$}, denoted by $\gcd(n,m)$. Then for any $n,m\in{^*}\BZ$, $\pr(n)\cap\pr(m)=\pr(\gcd(n,m))$. For $n,m \in {^*\BZ}$, if $n|m$, then $\pr(n)\subset \pr(m)$ but the converse does not need to hold. A filter of the form $\mathcal{F}(p)$ for some $p\in\CP({^*\BN})$ is maximal.
\end{remark}
\begin{definition}\label{filter_ideal}
\begin{enumerate}
	\item For an ideal $I$ in $^*\BZ$, define $\mathcal{F}(I):=\{\pr(n)|\ n\in I\}$.
	\item For a filter $\mathcal{F}$ on $\CP({^*\BN})$, define $I(\mathcal{F}):=\{n\in{^*\BZ}|\ \pr(n)\in \mathcal{F}\}$.
\end{enumerate}
\end{definition}
\begin{proposition}\label{basics_filter_ideal}
\begin{enumerate}
	\item For a proper ideal $I$ in $^*\BZ$, $\mathcal{F}(I)$ is a filter;
	\item For a filter $\mathcal{F}$ on $\CP({^*\BN})$, $I(\mathcal{F})$ is an ideal generated by elements $n$ in $^*\BZ$ such that
	\begin{equation*}
	\pr(n)\in\mathcal{F}\mbox{ and } d_p(n) = \twopartdef
{p}      {p\in\pr(n)}
{1}      {p\notin\pr(n)};
	\end{equation*}
	\item $\mathcal{F}(I_1)\subset\mathcal{F}(I_2)$ for two ideals $I_1\subset I_2$ in $^*\BZ$;
	\item $I(\mathcal{F}_1)\subset I(\mathcal{F}_2)$ for two filter $\mathcal{F}_1\subset \mathcal{F}_2$ on $\CP({^*\BN})$; and
	\item $I'\subset I(\mathcal{F}(I'))$ and $\mathcal{F}'\subset \mathcal{F}(I(\mathcal{F}'))$ for an ideal $I'$ in $^*\BZ$ and a filter $\mathcal{F}'$ on $\CP({^*\BN})$.
\end{enumerate}
\end{proposition}
\begin{proof}
(1) Let $I$ be a proper ideal in $^*\BZ$. Since $I$ is a proper ideal, so for all $a\in I$, $\pr(a)\neq \emptyset$. For $a,b\in I$, by Remark \ref{gcd}, $\pr(a)\cap\pr(b)=\pr(\gcd(a,b))$ and $\gcd(a,b)\in I$ so that $\pr(a)\cap\pr(b)\in \CF(I)$. Choose $a\in I$ and $b\in {^*\BZ}$ with $\pr(a)\subset \pr(b)$. Note that $\pr(b)=\pr(ab)$. Since $ab\in I$, $\pr(ab)$ is in $\CF(I)$ and so is $\pr(b)$. Thus $\CF(I)$ is a filter.\\

(2) Let $\CF$ be a filter on $\CP({^*\BN})$. Note that for any $a\in {^*\BZ}$, there is a unique $n_a$ such that $\pr(a)=\pr(n_a)$ and $d_p(n_a) = \twopartdef
{p}      {p\in\pr(n_a)}
{1}      {p\notin\pr(n_a)}$. So it is enough to show that $I(\CF)$ is an ideal. Take $a,b\in I(\CF)$ and $c\in {^*\BZ}$. Then $\pr(ca+b)=\pr(ca)\cap\pr(b)$. Since $\pr(a)\subset \pr(ca)$, the prime factor $\pr(ca+b)$ of $ca+b$ contains $\pr(a)\cap \pr(b)$. Since $\CF$ is a filter and $\pr(a)\cap \pr(b)\in \CF$, the prime factor $\pr(ca+b)$ is also in $\CF$ and thus $ca+b\in I(\CF)$. So $I(\CF)$ is an ideal.\\

(3) and (4) come from the definitions.\\

(5) Let $I'$ be an ideal in $^*\BZ$. Take $a\in I'$ arbitrary. The prime factor $\pr(a)$ of $a$ is in $\CF(I')$ by the definition. Again by definition of $I(\CF(I'))$, $a$ is in $I(\CF(I'))$. Therefore $I'$ is a subset of $I(\CF(I'))$. Let $\CF'$ be a filter on $\CP({^*\BN})$. Choose $a\in {^*\BN}$ such that its prime factor is in $\CF'$. In other word, $a$ is in $I(\CF')$. Then the prime factor $\pr(a)$ is in $\CF(I(\CF'))$. Thus $\CF'$ is a subset of $\CF(I(\CF'))$.
\end{proof}

\begin{theorem}\label{maximalfilter_maximalideal}
There is one-to-one correspondence between maximal ideals and maximal filters by the map $\fm\mapsto \mathcal{F}(\fm)$. Moreover, for a maximal ideal $\fm$, $\mathcal{F}(\fm)$ is principal if and only if $\fm=p{^*\BZ}$ for some $p\in\CP({^*\BN})$.
\end{theorem}
\begin{proof}
Note that for a maximal ideal $\fm$, $\fm\subseteq I(\mathcal{F}(\fm))$, and $1\notin I(\mathcal{F}(\fm))$ since $\emptyset\notin \mathcal{F}(\fm)$. By maximality of $\fm$, $\fm=I(\mathcal{F}(\fm))$ (*). For a maximal filter $\mathcal{F}'$, $\mathcal{F}'\subset \mathcal{F}(I(\mathcal{F}'))$, and $\emptyset \notin \mathcal{F}(I(\mathcal{F}'))$ since $1\notin I(\mathcal{F}')$. By maximality of $\mathcal{F}'$, $\mathcal{F}'=\mathcal{F}(I(\mathcal{F}'))$ (**). We claim that $I(\mathcal{F}')$ is a maximal ideal for a maximal filter $\mathcal{F}'$($\dagger$). Fix an maximal filter $\mathcal{F}'$. It is enough to show $^*\BZ/I(\mathcal{F}')$ is a field. Take $n+I(\mathcal{F}')$ with $n\notin I(\mathcal{F}')$. Then $\pr(n)$ is not in $\mathcal{F}'$. By maximality of $\mathcal{F}'$, there is $m\in I(\mathcal{F}')$ such that $\pr(m)\cap\pr(n)=\emptyset$ so that $I(\mathcal{F}')+n{^*}\BZ$. Thus there exists $k\in {^*\BZ}$ such that $nk+I(\mathcal{F}')=1+I(\mathcal{F}')$ and $k+I(\mathcal{F}')$ is the inverse of $n+I(\mathcal{F}')$. Therefore $^*\BZ/I(\mathcal{F}')$ is a field and $I(\mathcal{F}')$ is an maximal ideal.

We show this correspondence is injective. Let $\fm_1$ and $\fm_2$ be two maximal ideals. Suppose $\mathcal{F}(\fm_1)=\mathcal{F}(\fm_2)$. Since $\mathcal{F}(\fm_1)=\mathcal{F}(\fm_2)$, $I(\mathcal{F}(\fm_1))=I(\mathcal{F}(\fm_2))$ and  $\fm_1=\fm_2$ by (*). The surjectivity of the correspondence comes from (**) and ($\dagger$).

We now show the moreover part. Any principal maximal ideal is of the form of $p{^*\BZ}$ for some prime $p$ and the filter $\mathcal{F}_q$ is maximal for any prime $q$. Fix a prime $p$, $p{^*\BZ}\subset I(\mathcal{F}_p)$ and $\mathcal{F}_p\subset \mathcal{F}(p{^*\BZ})$. Thus $\mathcal{F}(p{^*\BZ})=\mathcal{F}_p$ and $I(\mathcal{F}_p)=p{^*\BZ}$. So, a maximal ideal $\fm$ is principal if and only if $\mathcal{F}(\fm)$ is principal. 
\end{proof}
We recall the classification of prime ideals in the ultrarpoducts of Dedekind domains in \cite{OS}.
\begin{fact}\cite{OS}\label{classification_OS}
Let $R$ be a Dedekind domain and let $^*R$ be an ultrapower of $R$. For each maximal ideal $\fm$ in $^*R$, the localization $^*R_{\fm}$ of $^*R$ at $\fm$ is a valuation domain and any prime ideal $\fp$ in $^*R$ is contained in only one maximal ideal.
\end{fact}
\noindent Next we describe prime ideals in $^*\BZ$ in terms of semisubgroups of valuation groups induced from each maximal ideals using Fact \ref{classification_OS}. Note that if $\fm$ is principal, then by Theorem \ref{maximalfilter_maximalideal}, $\fm=p{^*}\BZ$ for a prime $p$. So the valuation is $d_{p}$ with its valuation group $p^{^*\BZ}$, which is elementary equivalent to $(\BZ,+,<)$.
\begin{definition}
Let $(S,\cdot,<)$ be an ordered semigroup, that is, $(S,\cdot)$ is a semigroup and the binary relation $<$ is a totally linear order such that for $x,y,z\in S$, if $x<y$, then $z\cdot x<z\cdot y$ and $x\cdot z<y\cdot z$. A subset $T$ of $S$ is called a {\em subsemigroup} if $(T,\cdot)$ is a semigroup.
\begin{enumerate}
	\item A subsemigroup $T$ is {\em proper} if $T\neq S$.
	\item A subsemigroup $T$ is {\em convex} if for any $t_1<t_2\in T$ and $t_1<s<t_2$, $s$ is again in $T$. For a subsemigroup $T$, the convex hull $\langle T \rangle$ of $T$ is the smallest convex subsemigroup of $S$ containing $T$.
	\item A subsemigruop $T$ is {\em radical} if for any $x\in T$ and $n>0$, there is $y\in S$ such that $y^n=x$, then $y\in T$.
	\item A subsemigroup $T$ is {\em without right-end-point} if $T$ is convex and for each $x\in T$, if $y\in S$ satisfies $x<y$, then $y\in T$. For a subsemigroup $T$, denote by $\langle T\rangle_{\infty}$ the smallest subsemigroup of $S$ containing $T$ which is without right-end-point. 
	\item A subsemigroup $T$ is {\em prime} if for $a,b\in S$ with $ab\in T$, either $a\in T$ or $b\in T$.
\end{enumerate}
\end{definition}

\begin{remark}\label{prime_radical_semigp}
Let $(S,<)$ be an ordered semigroup. For $x\in S$ and $n\in \BZ_{>0}$, if there is $y\in S$ such that $y^n=x$, then such $y$ is unique. Let $T$ be a subsemigroup of $S$ without right-end-point. Then $T$ is radical if and only if $T$ is prime. Suppose $T$ is radical. Let $a\le b\in S$. Suppose $ab\in T$. We have that $b<ab<b^2$. Since $T$ is without right-end-point, $b^2\in T$ and by divisibility of $T$, $b$ is in $T$. Conversely, suppose $T$ is prime. If we have $x\in S$ and $n>0$ such that $x^n\in T$, then $x$ is in $T$ since $T$ is prime.
\end{remark}
\begin{example}
$(\BN,+,<)$ and $(\BQ_{\ge 0},+,<)$ are ordered semigroups. The set $(p^{^*\BN},\times,<)$ of power of $p$ is also an ordered semigroup , where for $x,y\in p^{^*\BN}$, $x<y$ if $y/x\in p^{^*\BN}$.
\end{example}
\begin{definition}\label{ideal_valuationring}
Let $(R,\nu)$ be a valuation ring. Consider $S:=\nu(R\setminus\{0\})$ as an ordered semigroup.
\begin{enumerate}
	\item For an ideal $I$ in $R$, define $S(I):=\nu(I)\cap S$.
	\item For a subsemigroup $T$ in $S$, define $I(T):=\bigcup_{n\in T}\nu^{-1}([n,\infty])$.

\end{enumerate}
\end{definition}
\begin{proposition}\label{basic_ideal_semigroup}
Let $(R,\nu)$ be a valuation ring and let $S:=\nu(R\setminus\{0\})$. Choose an ideal $I$ in $R$ and a subsemigroup $T$ of $S$.
\begin{enumerate}
	\item $S(I)$ is a subsemigroup without right-end-point.
	\item $I(T)$ is an ideal in $R$, and $I(T)=I(\langle T\rangle_{\infty})$.
	\item If $I=\sqrt{I}$, then $S(I)$ is radical.
\end{enumerate}
Moreover there is a one-to-one correspondence between ideals in $R$ and subsemigruops in $S$ without right-end-point by the map $I\mapsto S(I)$.
\end{proposition}
\begin{proof}
Let $(R,\nu)$ be a valuation ring and let $S:=\nu(R\setminus\{0\})$.\\

(1) Let $I$ be an ideal in $R$. Then $S(I)=\nu(I)\cap S$ itself is a semigroup. Choose $x_1\in S(I)$ and $x_2\in S$ such that $x_1<x_2$. Then there are $a_1,a_2\in R$ such that $\nu(a_1)=x_1$ and $\nu(a_2)=x_2$. Since $x_1<x_2$, $a_2/a_1\in R$ and $a_2=(a_2/a_1)a_1\in I$. Thus $x_2\in S(I)$.\\

(2) Note that for each $n\in S$, $\nu^{-1}([n,\infty])$ is an ideal in $R$. Let $T$ be a subsemigroup. Then $I(T)=\bigcup_{n\in T}\nu^{-1}([n,\infty])$ is a union of ideals in $R$ and it is an ideal. It remains to show $I(T)=I(\langle T\rangle_{\infty})$. We have that $I(T)\subseteq I(\langle T\rangle)$ since $T\subset \langle T\rangle_{\infty}$ It is enough to show that $I(T)\supseteq I(\langle T\rangle)$. Let $a\in I(\langle T\rangle_{\infty})$. Then $\nu(a)\in \langle T\rangle$. Thus there are $x\in \langle T\rangle_{\infty}$ such that $x\le \nu(a)$. This implies $a\in \nu^{-1}([x,\infty])$ and $a\in I(T)$.\\

(3) Suppose $I=\sqrt{I}$. Consider $x\in S$ such that $x^n\in S(I)$ for some $n>0$. Let $x=\nu(a)$ for some $a\in R$. Since $x^n\in S(I)$, it implies that $a^n\in I$, and thus $a\in I$.\\

Now we show the moreover part. First we show the given map is injective. Consider two ideals $I,J$ such that $S(I)=S(J)$. Let $a\in I$. Since $S(I)=S(J)$. There is $b\in J$ such that $\nu(a)=\nu(b)$, and this implies $a/b$ is a unit in $R$. Thus, $a=(a/b)b\in J$ and $I\subset J$. By the same way, $J$ is a subset of $I$ and therefore $I=J$. Next we show that the map is onto. Let $T$ be a subsemigroup of $S$ without right-end-point. Consider $S(I(T))$. It is clear that $T\subset S(I(T))$. Take $x\in S(I(T))$. By definition of $S(I(T))$, we have $a\in I(T)$ such that $\nu(a)=x$. Since $a\in I(T)$, for some $n,m\in T$, $n\le x\le m$. Since $T$ is convex, $x$ is in $T$. Thus $S(I(T))\subseteq T$ and we have that $T=S(I(T))$.
\end{proof}
\begin{remark}\label{primeideal_localization}
Let $R$ be a commutative ring and $\fm$ be a maximal ideal in $R$. Then there is one-to-one correspondence between prime ideals of $R$ contained in $\fm$ and prime ideals in the localization $R_{\fm}$ of $R$ at $\fm$.
\end{remark}
\noindent Let $\fm$ be a maximal ideal in $^*\BZ$. Then the localization $^*\BZ_{\fm}$ is a valuation ring by Fact \ref{classification_OS}. Let $\nu_{\fm}$ be the valuation of $^*\BZ_{\fm}$. We write $S_{\fm}$ for the ordered semigroup $\nu_{\fm}(^*\BZ_{\fm}\setminus\{0\})$.
\begin{theorem}\label{primeideal_semigroup}
For a given maximal ideal $\fm$ of $^*\BZ$, there is a one-to-one correspondence between prime ideals contained in $\fm$ and proper radical subsemigroup of $S_{\fm}$ without right-end-point by the map $P\mapsto \nu_{\fm}((P{^*\BZ}_{\fm})\setminus\{0\})$.
\end{theorem}
\begin{proof}
Choose a maximal ideal $\fm$ of $^*\BZ$. By Theorem \ref{classification_OS}, the localization $^*\BZ_{\fm}$ of $^*\BZ$ at $\fm$ is a valuation ring. Apply Remark \ref{primeideal_localization} and Proposition \ref{basic_ideal_semigroup} to the valuation ring $^*\BZ_{\fm}$. Note that any prime ideal is itself radical.
\end{proof}
\begin{example}
There is only one proper radical subsemigroup without right-end-point(called PRwE semigroup in short), $\BN_{>0}$ and $\BQ_{>0}$ in $\BN$ and $\BQ_{\ge 0}$ respectively. For each prime $p$ in $^*\BZ$, there are infinitely many $x_0<x_1<x_2<\ldots$ in $d_p({^*}\BZ)=p^{^*\BN}$ such that all convex hulls $T_i=\{x_i^1,x_i^2,\ldots\}$ of $x_i$'s are disjoint, and they give infinitely many different PRwE semigroups. Thus there are infinitely many prime ideals in $p{^*\BZ}$.
\end{example}

\section{Appendix}
Here, we extend Remark \ref{ultraproduct_abeliangroup} to definable abelian groups in saturated nonstandard models of $\Th(\BQ)$ in the ring language $\CL_{ring}=\{+,-,\times;0,1\}$. Let $^*\BQ$ be a saturated model of $\Th(\BQ)$ and $^*\BZ$ be a nonstandard integer ring of $^*\BZ$ corresponding to $\BZ$ in $\BQ$, which is definable. Let $(A,+_A,1_A)$ be an definable abelian group in $^*\BQ$. It is clear that $A$ is a $\BZ$-module as 
$$
n\cdot a:=
\begin{cases}
\overbrace{a+_A\cdots+_A a}^{|n|}&\mbox{if } n>0\\
1_A&\mbox{if } n=0\\
\underbrace{(-a)+_A\cdots+_A (-a)}^{|n|}&\mbox{if } n<0
\end{cases}
$$
for $a\in A$ and $n\in \BZ$. Our main aim in this appendix is to show that we can see $A$ as $^*\BZ$-module extending the natural $\BZ$-module structure, that is, there is a map $\cdot^*:{^*\BZ}\times A\rightarrow A$ to make $A$ as $^*\BZ$-module and $n\cdot^* a=n\cdot a$ for $a\in A$ and $n\in \BZ$.

Let $V$ be the standard model of ZFC axiom in the language $\CL_{set}=\{\in\}$. We know that $\BZ$ and $\BQ$ are definable with the ring and field structures of $\BZ$ and $\BQ$ definable in the language $\CL_{set}$ and let $\phi_{\BZ}(x)$ and $\phi_{\BQ}(x)$ be such formulas with respect to $\BZ$ and $\BQ$. Consider a definable group $(G,+_G,1_G)$ and let $\theta(\x,\y)$ be a formula such that for a parameter set $\ba$, the formula $\theta(\x,\ba)$ defines the set $G$ and says the function $+_G$ on $G$ is a group operation with the identity $1_G$. Let $\CC_\theta(\y)$ be a formula saying that for $\bb\in V^{|\y|}$, $\CC_\theta(\bb)$ holds if and only if $\theta(\x,\bb)$ defines a group $(G_{\bb},+_{G_{\bb}},1_{G_{\bb}})$. Consider a formula $\CF_\theta(\x,\y)\equiv \CC_\theta(\y)\wedge \theta(\x,\y)$ parametrizing all pairs of groups definable by the formula $\theta$ and points in such groups. Consider a function $f$ from $\BZ\times \CF_\theta(V)\rightarrow \CF_\theta(V)$ such that $f(0,\g,\bb)=(1_{G_{\bb}},\bb)$ and $f(n+1,\g,\bb)=(\g+_{G_{\bb}} \pi(f(n,\g,\bb)),\bb)$ for $n\in \BZ$ and $(\g,\bb)\in \CF_{\theta}(V)$, where $\pi$ is the projection map from $\CF_{\theta}(V)$ to $\bigcup_{\bb\in \CC_{\theta}(V)} G_{\bb}$. Then by recursion theorem, this function $f$ is definable. Moreover if $G_{\bb}$ is commutative for $\bb \in V^{|\y|}$, then the $\BZ$-module structure on $G_{\bb}$ is definable by $f(\cdot,\cdot,\bb)$.

Now let $\kappa$ be an inaccessible cardinal and let $^*V$ be a saturated extension of $V$ of cardinality $\kappa$. Let $^*\BZ=\phi_{\BZ}({^*V})$ and $^*\BQ=\phi_{\BQ}({^*V})$ which are saturated extensions of $\BZ$ and $\BQ$ of cardinality $\kappa$. Then any definable abelian group $({^*G},+_{^*G},1_{^*G})$ has a $^*\BZ$-module structure extending the $\BZ$-module structure. Since any saturated model is unique up to isomorphism, we get the following result : Let $\kappa$ be an inaccessible cardinal. Let $^*\BQ$ be a saturated model of $\Th({^*\BQ})$ of cardinality $\kappa$ and let $^*\BZ$ be a nonstandard integer ring corresponding to $\BZ$ in $\BQ$.

\begin{theorem}\label{nonstandard-module}
Any abelian group definable in $^*\BQ$ has a $^*\BZ$-module structure extending the $\BZ$-module structure. Specially any elliptic curves over $^*\BQ$ has a $^*\BZ$-module structure.
\end{theorem}
\noindent As a corollary, we extend Theorem \ref{rkbound_wMW} to saturated models of $\Th(\BQ)$.
\begin{theorem}\label{rkbound-saturatedmodels}
The following are equivalent :
\begin{enumerate}
	\item The ranks of elliptic curves over $\BQ$ are uniformly finitely bounded.
	\item For each $n\ge 2$, the cardinalities of weak $n$th Mordell-Weil groups over $\BQ$ are uniformly finitely bounded.
	\item The cardinalities of weak $2$nd Mordell-Weil groups over $\BQ$ are uniformly finitely bounded.
	\item Weak Mordell-Weil property holds for $^*\BQ$.
	\item Weak $2$nd Mordell-weil groups of elliptic curves over $^*\BQ$ are finite.
	\item Nonstandard Moredell-Weil property hold for $^*\BQ$.
\end{enumerate}
\end{theorem}

\bibliographystyle{amsplain}


\end{document}